\documentclass[12pt]{extarticle}

\usepackage[margin=1.5in]{geometry}  
\usepackage{graphicx}              
\usepackage{amsmath}               
\usepackage{amsfonts}              
\usepackage{amsthm}                
\usepackage{framed}
\usepackage{verbatim}
\usepackage{url}

\newtheorem{thm}{Theorem}
\newtheorem{lemma}[thm]{Lemma}

\begin{document}

\nocite{*}

\title{\bf On the Sum of the Square of a Prime and a Square-Free Number}

\author{\textsc{Adrian W. Dudek} \\ 
Mathematical Sciences Institute \\
The Australian National University \\ 
\texttt{adrian.dudek@anu.edu.au}\\
\\
\textsc{David J. Platt} \\ 
 Heilbronn Institute for Mathematical Research\\
University of Bristol, Bristol, UK \\ 
\texttt{dave.platt@bris.ac.uk}}
\date{}

\maketitle

\begin{abstract}
\noindent We prove that every integer $n \geq 10$ such that $n \not\equiv 1 \text{ mod } 4$ can be written as the sum of the square of a prime and a square-free number.  This makes explicit a theorem of Erd\H{o}s that every sufficiently large integer of this type may be written in such a way. Our proof requires us to construct new explicit results for primes in arithmetic progressions. As such, we use the second author's numerical computation regarding GRH to extend the explicit bounds of Ramar\'{e}--Rumely.
\end{abstract}

\section{Introduction}

We say that a positive integer is square-free if it is not divisible by the square of any prime number. It was proven by Erd\H{o}s \cite{erdos} in 1935 that every sufficiently large integer $n \not\equiv 1 \text{ mod } 4$ may be written as the sum of the square of a prime and a square-free number. The congruence condition here is sensible. If $n\equiv 1 \mod 4$ then $4|(n-p^2)$ for any odd prime $p$. This only leaves the case $p=2$, but $n-4$ fails to be square-free infinitely often\footnote{For example, one can consider the congruence class $13 \text{ mod } 36$.}.

It is the objective of this paper to make explicit the proof provided by Erd\H{o}s, to the end of proving the following theorem. 

\begin{thm} \label{main}
Let $n \geq 10$ be an integer such that $n \not\equiv 1 \mod 4$. Then there exists a prime $p$ and a square-free number $k$ such that $n=p^2+k$.
\end{thm}

In a recent paper \cite{dudeksum}, the first author proved that every integer greater than two can be written as the sum of a prime and a square-free number. One can think of such a result as a weak-but-explicit form of Goldbach's conjecture. Theorem \ref{main} is significantly stronger than this, for the sequence of squares of primes is far more sparse than the sequence of the primes. To prove Theorem \ref{main}, we combine modern explicit results on primes in arithmetic progressions and computation. 

The proof may be outlined as follows. For any integer $n$ satisfying the conditions of the above theorem, we want to show that there exists a prime $p < \sqrt{n}$ such that $n-p^2$ is square-free. That is, we require some prime $p$ such that
$$n-p^2 \not\equiv 0 \mod q^2$$
for all odd primes $q < \sqrt{n}$. The idea is to consider, for some large $n$ and each odd prime $q< \sqrt{n}$, those \textit{mischievous} primes $p$ that satisfy the congruence
$$n \equiv p^2 \mod q^2.$$
Then, for each $q$ we explicitly bound from above (with logarithmic weights) the number of primes $p$ which satisfy the above congruence. Summing over all moduli $q$ gives us an upper bound for the weighted count of the so-called mischievous primes
$$\sum_{q < \sqrt{n}} \sum_{\substack{p < \sqrt{n} \\ n \equiv p^2 \text{ mod } q^2}} \log p.$$
It is then straightforward to show that for large enough $n$, the above sum is less than the weighted count of \emph{all} primes less than $\sqrt{n}$, and therefore there must exist a prime $p < \sqrt{n}$ such that $n-p^2$ is not divisible by the square of any prime. 

This method works well, and allows us to prove Theorem \ref{main} for all integers $n \geq 2.5 \cdot 10^{14}$ that satisfy the congruence condition. We eliminate the remaining cases by direct computation to complete the proof.

\section{Theorem \ref{main} for large integers}

\subsection{Case 1}

We start by considering integers in the range $n \geq 2.5\cdot 10^{14}$ such that $n \not \equiv 1 \text{ mod } 4$. As usual, we define
$$\theta(x;k,l) = \sum_{\substack{p \leq x \\ p \equiv l \text{ mod } k}} \log p,$$
where $p$ denotes a prime number. 

The paper of Ramar\'{e}--Rumely \cite{ramarerumely} provides us with bounds of the form
\begin{equation*}
\bigg| \theta(x;k,l) - \frac{x}{\varphi(k)} \bigg| < \epsilon(k,x_0) \frac{x}{\varphi(k)}
\end{equation*}
and
\begin{equation*}
\bigg| \theta(x;k,l) - \frac{x}{\varphi(k)} \bigg| < \omega(k,x_1) \sqrt{x}.
\end{equation*}
for various ranges of $x\geq x_0$ and $x\leq x_1$ respectively. These computations were in turn based on Rumely's numerical verification of the Generalised Riemann Hypothesis (GRH) \cite{Rumely1993} for various moduli and to certain heights. Since then, the second author has verified GRH for a wider range of moduli and to greater heights \cite{Platt2013}. For our purposes, we rely only on the following:
\begin{lemma}\label{grh}
Let $q$ be a prime satisfying $17 \leq q \leq 97$. All non-trivial zeros $\rho$ of Dirichlet L-functions derived from characters of modulus $q^2$ with $\Im\rho \leq 1000$ have $\Re\rho =1/2$.
\end{lemma} 
\begin{proof}
See Theorem 10.1 of \cite{Platt2013}.
\end{proof}

We can therefore extend the results of Ramar\'{e}--Rumely with the following lemma:

\begin{lemma}
For $x>10^{10}$ we have 
\begin{equation*}
\bigg| \theta(x;q^2,l) - \frac{x}{\varphi(q^2)} \bigg| < \epsilon(q^2,10^{10}) \frac{x}{\varphi(q^2)}
\end{equation*}
for the values of $q$ and $\epsilon(q^2,10^{10})$ in Table \ref{tab:new_eq}.
\end{lemma}
\begin{proof}
We refer to \cite{ramarerumely}. The values for $q\in\{3,5,7,11,13\}$ are from Table 1 of that paper. For the other entries, we use Theorem 5.1.1 with $H_\chi=1000$ and $C_1(\chi,H_\chi)=9.14$ (see display 4.2). We set $m=10$ for $q\leq 23$, $m=12$ for $q\geq 47$ and $m=11$ otherwise. We use $\delta=2e/H_\chi$ and for $\widetilde{A}_\chi$ we use the upper bound of Lemma 4.2.1. Finally, for $\widetilde{E}_\chi$ we rely on Lemma 4.1.2 and we note that $2\cdot 9.645908801\cdot\log^2(1000/9.14)\geq \log 10^{10}$ as required.
\end{proof}

\begin{table}[h!] 
\caption{Values for $\epsilon(q^2,10^{10})$.} 
\label{tab:new_eq} 
\centering 
\begin{tabular}{| c c | c c | c c | c c |} 
\hline 
$q$ & $\epsilon(q^2,10^{10})$ & $q$ & $\epsilon(q^2,10^{10})$ & $q$ & $\epsilon(q^2,10^{10})$ & $q$ & $\epsilon(q^2,10^{10})$  \\[0.5 ex] \hline
$3$ & $0.003228$ & $19$ & $0.17641$ & $43$ & $0.95757$ & $71$ & $2.82639$\\
$5$ & $0.012214$ & $23$ & $0.25779$ & $47$ & $1.15923$ & $73$ & $3.00162$\\
$7$ & $0.017015$ & $29$ & $0.41474$ & $53$ & $1.50179$ & $79$ & $3.56158$\\
$11$ & $0.031939$ & $31$ & $0.47695$ & $59$ & $1.89334$ & $83$ & $3.96363$\\
$13$ & $0.042497$ & $37$ & $0.69397$ & $61$ & $2.03488$ & $89$ & $4.61023$\\
$17$ & $0.14271$ & $41$ & $0.86446$ & $67$ & $2.49293$ & $97$ & $5.55434$\\
\hline 
\end{tabular} 
\end{table} 

\begin{lemma}
We have 
$$\omega\left(3^2,10^{10}\right)=1.109042,$$ 
$$\omega\left(5^2,10^{10}\right)=0.821891,$$ 
$$\omega\left(7^2,10^{10}\right)=0.744132,$$
$$\omega\left(11^2,10^{10}\right)=0.711433$$ 
and 
$$\omega\left(13^2,10^{10}\right)=0.718525.$$ 
If $q$ is a prime such that $17 \leq q \leq 97$ we have
\begin{equation*}
\omega(q^2,10^{10})=\frac{\log 7 - \frac{7}{\varphi(q^2)}}{\sqrt{7}}.
\end{equation*}
\end{lemma}
\begin{proof}
The results for $\{3^2,5^2,7^2,11^2,13^2\}$ are from Table 2 of \cite{ramarerumely} with a slight correction to the entry for $5^2$. A short computation shows that the maximum occurs for all of the other $q$ when $x=7$ and $a=7$.
\end{proof}

\begin{lemma} \label{apbounds}
Let $T=\sqrt{2.5\cdot 10^{14}}$. Then for $x\geq T$ and $q\leq 97$ an odd prime we have
$$\bigg| \theta(x;q^2,l) - \frac{x}{\varphi(q^2)} \bigg| < \epsilon\left(q^2,T\right) \frac{x}{\varphi(q^2)},$$
where the values of $\epsilon\left(q^2,T\right)$ are given in Table \ref{table1}.
\end{lemma}

\begin{proof}
Using $\omega(q^2,10^{10})$ we have
\begin{equation*}
\bigg| \theta(T;q^2,l) - \frac{T}{\varphi(q^2)} \bigg| < \omega\left(q^2,10^{10}\right)\sqrt{T}
\end{equation*}
so for $x\in[T,10^{10}]$ we have
\begin{equation*}
\bigg| \theta(x;q^2,l) - \frac{x}{\varphi(q^2)} \bigg| < \frac{\omega\left(q^2,10^{10}\right)\varphi(q^2)}{\sqrt{T}}\frac{x}{\varphi(q^2)}
\end{equation*}
and so we can take 
\begin{equation*}
\epsilon\left(q^2,T\right)=\max\left(\epsilon\left(q^2,10^{10}\right),\frac{\omega\left(q^2,10^{10}\right)\varphi(q^2)}{\sqrt{T}}\right).
\end{equation*}
\end{proof}

\begin{table}[h!] 
\caption{Values for $\epsilon(q^2,T)$ for Lemma \ref{apbounds}.} 
\label{table1} 
\centering 
\begin{tabular}{c c c c c c c c } 
\\ \hline\hline 
$q$ & $\epsilon(q^2,T)$ & $q$ & $\epsilon(q^2,T)$ & $q$ & $\epsilon(q^2,T)$ & $q$ & $\epsilon(q^2,T)$  \\[0.5 ex] \hline
$3$ & $0.00323$ & $19$ & $0.17641$ & $43$ & $0.95757$ & $71$ & $2.82639$\\
$5$ & $0.01222$ & $23$ & $0.25779$ & $47$ & $1.15923$ & $73$ & $3.00162$\\
$7$ & $0.01702$ & $29$ & $0.41474$ & $53$ & $1.50179$ & $79$ & $3.56158$\\
$11$ & $0.03194$ & $31$ & $0.47695$ & $59$ & $1.89334$ & $83$ & $3.96363$\\
$13$ & $0.04250$ & $37$ & $0.69397$ & $61$ & $2.03488$ & $89$ & $4.61023$\\
$17$ & $0.14271$ & $41$ & $0.86446$ & $67$ & $2.49293$ & $97$ & $5.55434$\\
\hline\hline 
\end{tabular} 
\end{table}

Let $n \geq 2.5\cdot 10^{14}$ be such that $n \not\equiv 1 \text{ mod } 4$ and consider the case where $q$ is an odd prime such that $q \leq 97$. We want to bound from above the number of primes $p < \sqrt{n}$ satisfying
\begin{equation} \label{congruence}
n \equiv p^2 \text{ mod } q^2.
\end{equation}
Clearly, $p$ can belong to at most two arithmetic progressions moduluo $q^2$. Therefore, by Lemma \ref{apbounds},  we can estimate the weighted count of such primes as follows.
$$\sum_{\substack{p < \sqrt{n} \\ n \equiv p^2 \text{ mod } q^2}} \log p \leq  \theta(\sqrt{n}; q^2, l) + \theta(\sqrt{n}; q^2, l')  < \frac{2 (1+\epsilon(q^2,T))}{q (q-1)} \sqrt{n}$$
where $l$ and $l'$ are the possible congruence classes for $p$ and $\epsilon(q^2,T)$ is given in Table \ref{table1}. Summing this over all $24$ values of $q$ gives us the contribution

\begin{equation} \label{cont1}
\sum_{q \in \{3,\ldots,97\}} \sum_{\substack{p < \sqrt{n} \\ n \equiv p^2 \text{ mod } q^2}} \log p <0.568 \sqrt{n}.
\end{equation}

\subsection{Case 2}

We now consider the case where $97 < q \leq n^{c}$ and $c \in (0,1/4)$ is to be chosen later to achieve an optimal result. Montgomery and Vaughan's \cite{montgomeryvaughan} explicit version of the Brun-Titchmarsh Theorem gives us that
$$\pi(x;k,l) \leq \frac{2 x}{\varphi(k) \log(x/k)}$$
for all $x>q$. Trivially, one has that
$$\theta(\sqrt{n}; q^2,l) \leq \frac{\sqrt{n}}{q(q-1)} \frac{ \log n}{\log(\sqrt{n}/q^2)}.$$
As $q < n^{c}$, it follows that
\begin{equation} \label{cont1point5}
\sum_{97 < q \leq n^{c}} \sum_{\substack{p < \sqrt{n} \\ n \equiv p^2 \text{ mod } q^2}} \log p < \frac{\sqrt{n}}{\frac{1}{4}- c}  \sum_{97 < q \leq n^{c}} \frac{1}{q(q-1)}.
\end{equation}
We can bound the sum as follows:
\begin{eqnarray*}
\sum_{97 < q \leq n^{c}} \frac{1}{q(q-1)} & < & \sum_{97 < q < 1000001} \frac{1}{q (q-1)} + \sum_{ n \geq 1000001} \frac{1}{n(n-1)} \\
& = & \sum_{97 < q <1000001} \frac{1}{q (q-1)} + \frac{1}{1000000} <0.00183.
\end{eqnarray*}
Substituting this into (\ref{cont1point5}) gives us that
\begin{equation} \label{cont2}
\sum_{97 < q \leq n^{c}} \sum_{\substack{p < \sqrt{n} \\ n \equiv p^2 \text{ mod } q^2}} \log p < \frac{0.00183 \sqrt{n}}{\frac{1}{4}-c}.
\end{equation}

\subsection{Case 3}
 
Let $q$ be an odd prime such that $n^{c} < q < A \sqrt{n}$ and $A \in (0,1)$ is to be chosen later for optimisation. Since there are at most two possible residue classes modulo $q^2$ for $p$, the number of primes $p$ such that $n \equiv p^2 \text{ mod } q^2$ is trivially less than
$$2\bigg( \frac{\sqrt{n}}{q^2} + 1\bigg).$$
Clearly, including our logarithmic weights one has that
\begin{equation*}
\sum_{\substack{p < \sqrt{n} \\ n \equiv p^2 \text{ mod } q^2}} \log p <  \bigg( \frac{\sqrt{n}}{q^2} + 1\bigg) \log n
\end{equation*}
and so
\begin{eqnarray*}
\sum_{n^{c} < q < A \sqrt{n}} \sum_{\substack{p < \sqrt{n} \\ n \equiv p^2 \text{ mod } q^2}} \log p < \sqrt{n} \log n \sum_{m > n^{c}} \frac{1}{m^2} + \pi(A \sqrt{n}) \log(n)
\end{eqnarray*}
where $\pi(x)$ denotes the number of primes not exceeding $x$. The sum can be estimated in a straightforward way by 
$$\sum_{m > n^{c}} \frac{1}{m^2} < \frac{1}{n^{2c}} + \int_{n^{c}}^{\infty} \frac{1}{t^2} dt = \frac{1}{n^{2c}}+\frac{1}{n^c}$$
and Theorem 6.9 of Dusart \cite{dusart} gives us that
$$\pi(A \sqrt{n}) < \frac{A \sqrt{n}}{\log(A \sqrt{n})}\bigg(1+\frac{1.2762}{\log (A \sqrt{n})}\bigg).$$
Therefore, putting this all together we have
\begin{equation} \label{cont3}
\sum_{n^c<q<A \sqrt{n}} \sum_{\substack{p < \sqrt{n} \\ n \equiv p^2 \text{ mod } q^2}} \log p < \sqrt{n}(n^{-2c}+n^{-c})\log n + \frac{A \sqrt{n} \log n}{\log(A \sqrt{n})}\bigg(1+\frac{1.2762}{\log (A \sqrt{n})}\bigg).
\end{equation}

\subsection{Case 4}

Finally, we consider the range $A \sqrt{n} \leq q < \sqrt{n}$. If $n-p^2$ is divisible by $q^2$, then
\begin{equation} \label{blah}
n = p^2 + B q^2
\end{equation}
for some positive integer $B < A^{-2}$. We will need some preliminary results here. First, it is known by the theory of quadratic forms (see Davenport \cite[Ch. 6]{davenport}) that the equation
$$ax^2+by^2=n,$$
where $a, b$ and $n$ are given positive integers, has at most $w 2^{\omega(n)}$ proper solutions, that is, solutions with $\gcd(x,y)=1$. Note that $w$ denotes the number of automorphs of the above form and $\omega(n)$ denotes the number of different prime factors of $n$. The number of automorphs is directly related to the discriminant of the form; specifically, $w=4$ for the case $B=1$ and $w=2$ for $B>1$. Moreover, we are only interested in the case where $x$ and $y$ are both positive, and so it follows that equation (\ref{blah}) has at most $w 2^{\omega(n)-2}$ proper solutions. Finally, noting that there will be at most 1 improper solution to (\ref{blah}), namely $p=q$, we can bound the overall number of solutions to (\ref{blah}) by $w 2^{\omega(n)-2}+1$.

Furthermore, Theorem 11 of Robin \cite{robin}  gives us the explicit bound
$$\omega(n) \leq 1.3841 \frac{\log n}{\log \log n} $$
for all $n \geq 3$. Thus, for fixed $n$ and $B$, it is easy to bound explicitly from above the number of solutions to (\ref{blah}). It remains to sum this bound over all valid values of $B$. However, we should note that given an integer $n$, there are not too many good choices of $B$, and this will allow us to make a further saving.

This comes from the observation that every prime $p >3$ satisfies $p^2 \equiv 1 \text{ mod } 24$. For with $p>3$ and $q>3$, Equation (\ref{blah}) becomes
$$B \equiv n - 1 \text{ mod } 24,$$
and this confines $B$ to the integers in a single residue class modulo 24. 

Formally and explicitly, we argue as follows. Consider first the case where $B$ is an integer in the range
$$\frac{n-9}{A^2 n} \leq B < \frac{1}{A^2}.$$
The leftmost inequality above keeps $p \leq 3$. Here, there are clearly at most 
$$\frac{9}{A^2 n} + 1$$
integer values for $B$. We now consider the case where $p>3$, and it follows that $B \equiv n - 1 \text{ mod } 24$. Clearly, then, there are at most 
$$\frac{1}{24 A^2} + 1$$ values for $B$ in this range. Therefore, in total, there are at most 
$$2 + \frac{1}{24 A^2} + \frac{9}{A^2 n}$$
values of $B$ for which we need to sum the solution counts to Equation (\ref{blah}). Also, we must also consider that $w=4$ for $B=1$. Therefore, we have that the number of solutions to Equation (\ref{blah}) summed over $B$ is bounded above by
$$2^{\omega(n)-1}\bigg(3 + \frac{1}{24 A^2} + \frac{9}{A^2 n}\bigg).$$
Therefore, the number of primes $p$ (including weights) which satisfy (\ref{blah}) is at most
\begin{equation} \label{cont4}
\sum_{A \sqrt{n} \leq q < \sqrt{n}} \sum_{\substack{p < \sqrt{n} \\ n \equiv p^2 \text{ mod } q^2}} \log p <  2^{1.3841 \log n / \log \log n}\bigg(\frac{3}{2} + \frac{1}{48 A^2} + \frac{9}{2A^2 n}\bigg) \log n.
\end{equation}

\subsection{Collecting terms}

Now, collecting together (\ref{cont1}), (\ref{cont2}), (\ref{cont3}) and (\ref{cont4}), we have that the weighted count over all the so-called mischevious primes can be bounded thus
\begin{eqnarray*}
\sum_{q < \sqrt{n}} \sum_{\substack{p < \sqrt{n} \\ n \equiv p^2 \text{ mod } q^2}} \log p & < & \bigg(0.568+\frac{0.00183}{\frac{1}{4}-c}+ (n^{-2c}+n^{-c})\log n \bigg) \sqrt{n} \\
& &+ \;\frac{A \sqrt{n} \log n}{\log(A \sqrt{n})}\bigg(1+\frac{1.2762}{\log (A \sqrt{n})}\bigg) \\
& &+ \;2^{1.3841 \log n / \log \log n}\bigg(\frac{3}{2} + \frac{1}{48 A^2} + \frac{9}{2A^2 n}\bigg) \log n.
\end{eqnarray*}
As expected, however, the weighted count over all primes exceeds this for large enough $n$ and good choices of $c$ and $A$. Dusart \cite{dusart} gives us that 
$$\theta(x) \geq x - 0.2 \frac{x}{\log^2 x}$$
for all $x \geq 3 594 641$, and thus it follows that
$$\theta(\sqrt{n}) \geq \sqrt{n} - 0.8 \frac{\sqrt{n}}{\log^2 n}$$
for all $n \geq 10^{14}$. Therefore, if we denote by $R(n)$ the (weighted) count of primes $p$ such that $n-p^2$ is square-free, it follows that
\begin{eqnarray*}
R(n) & > & \bigg(1-0.568-\frac{0.00183}{\frac{1}{4}-c}-\frac{0.8}{\log^2 n} -(n^{-2c}+n^{-c})\log n\bigg) \sqrt{n}\\
& - &\frac{A \sqrt{n} \log n}{\log(A \sqrt{n})}\bigg(1+\frac{1.2762}{\log (A \sqrt{n})}\bigg) \\
& - & 2^{1.3841 \log n / \log \log n}\bigg(\frac{3}{2} + \frac{1}{48 A^2} + \frac{9}{2A^2 n}\bigg) \log n.
\end{eqnarray*}
It is now straightforward to check that choosing $c=0.209$ and $A=0.0685$ gives $R(n)>0$ for all $n \geq 2.5 \times 10^{14}$.

\section{Numerical Verification for ``Small'' $n$}

We now describe a computation undertaken to confirm that all $n\not\equiv 1 \mod 4$, $10\leq n \leq 4\,000\,023\,301\,851\,135$ can be written as the sum of a prime squared and a square-free number.\footnote{This is a factor of $16$ further than we actually needed to check, but we did not expect our analytic approach to fare as well as it did.} We will first describe the algorithm used, and then say a few words about its implementation.

\subsection{The Algorithm}

We aim to test $3\cdot 10^{15}$ different $n$. We quickly conclude that we cannot afford to individually test candidate $n-p^2$ to see if they are square-free. There is an analytic algorithm \cite{Booker2015} that is conjectured to be able to test a number of size $n$ in time $\mathcal{O}(\exp([\log n]^{2/3+o(1)})$ but this is contingent on the Generalised Riemann Hypothesis. We would be left needing to factor each $n-p^2$, which would be prohibitively expensive.

We proceeed instead by chosing a largest prime $P$ and a sieve width $W$. To check all the integers in $[N,N+W)$ we first sieve all the integers in $[N-P^2,N+W-4)$ by crossing out any that are divisible by a prime square $p^2$ with $p<\sqrt{(N+W-5)/2}$. Now for each $n\in[N,N+W)$, $n\not\equiv 1 \mod 4$ we lookup in our sieve to see if $n-4$ is square-free\footnote{Unless $n\equiv 0 \mod 4$.}. If not, we try $n-9$ then $n-25$ and so on until $n-p^2$ is square-free. If it fails all these tests up to and including $n-P^2$, we output $n$ for later checking.

\subsection{The Implementation}

Numbers of this size fit comfortably in the $64$ bit native word size of modern CPUs and we implemented the algorithm in C++. We use a character array for the sieve\footnote{We considered using each byte to represent $8$ or more $n$ but the cost of the necessary bit twiddling proved too heavy.}, and chose a sieve width $W=2^{31}$ as this allows us to run $16$ such sieves in parallel in the memory available. We set the prime limit $P=43$ as this was found to reduce the number of failures to a manageable level (see below). To generate the primes used to sieve the character array we used Kim Walisch's PrimeSieve \cite{Walisch2012}.

We were able to run $16$ threads on a node of the University of Bristol's Bluecrystal Phase III cluster \cite{ACRC2014} and in total we required $5,400$ core hours of CPU time to check all $n\in[2048,4\,000\,023\,301\,851\,135]$. $4\,915$ $n$ were rejected as none of $n-p^2$ with $p\leq 43$ were square-free. We checked these $4\,915$ cases in seconds using PARI \cite{Batut2000} and found that $p=47$ eliminated $4\,290$ of them, $53$ does for a further $538$, $59$ for $14$ more, $61$ for $61$ (!), $67$ doesn't help (!), $71$ kills off $11$ more and the last one standing, $n=1\,623\,364\,493\,706\,484$ falls away with $p=73$. Finally, we use PARI again to check $n\in[10,2047]$ with $n\not\equiv 1 \mod 4$ and we are done.

It is interesting to consider the efficiency of the main part of this algorithm. The CPUs on the compute nodes of Phase III are $2.6$GHz Intel\textsuperscript{\textregistered} Xeon\textsuperscript{\textregistered} processors and we checked $3\cdot 10^{15}$ individual $n$ in $5\,400$ hours. This averages less than $17$ clock ticks per $n$ which suggests that the implementation must have made good use of cache.

\clearpage

\bibliographystyle{plain}

\bibliography{biblio}

\end{document}